\documentclass[11pt]{amsart}
\usepackage{amsfonts}

\usepackage{color,comment,fullpage}

\usepackage{array}

\usepackage{amsfonts,amsmath, amssymb}

\newcommand{\bea}{\begin{eqnarray}}

\newcommand{\eea}{\end{eqnarray}}

\newcommand{\be}{\begin {equation}}

\newcommand{\ee}{\end{equation}}

\newcommand{\N}{\Bbb N}
\newcommand{\Z}{\mathbb{Z}}

\newtheorem{theorem}{Theorem}[section]

\newtheorem{corollary}[theorem]{Corollary}
\newtheorem{lemma}[theorem]{Lemma}
\newtheorem{proposition}[theorem]{Proposition}
\newtheorem{remark}[theorem]{Remark}

\newtheorem{conjecture}[theorem]{Conjecture}

\newcommand{\am}{\mathcal{W}(p)^{A_m}}
\newcommand{\triplet}{\mathcal{W}(p)}
\begin{document}

\title{$ADE$ subalgebras of the triplet vertex algebra $\triplet$: $A$-series}
\author{Dra\v zen Adamovi\' c}
\address{Department of Mathematics, University of Zagreb, Bijenicka 30, 10000 Zagreb, Croatia}
\email{adamovic@math.hr}
\author {Xianzu Lin }
\address{College of Mathematics and Computer Science, Fujian Normal University, Fuzhou, {\rm 350108}, China}
\email{linxianzu@126.com}
\author{Antun Milas}
\address{Department of Mathematics and Statistics,SUNY-Albany, 1400 Washington Avenue, Albany 12222,USA}
\email{amilas@albany.edu}
\date{ }
\maketitle

\begin{abstract}

Motivated by \cite{am1}, for every finite subgroup $\Gamma \subset PSL(2,\mathbb{C})$  we investigate the fixed point subalgebra $\triplet^{\Gamma}$ of the triplet vertex $\mathcal {W}(p)$, of central charge $1-\frac{6(p-1)^{2}}{p}$, $p\geq2$.
This part deals with the $A$-series in the ADE classification of finite subgroups of $PSL(2,\mathbb{C})$. First, we prove the $C_2$-cofiniteness of the $A_m$-fixed subalgebra $\triplet^{A_m}$. Then we construct a family of $\am$-modules, which are expected to form a complete set of irreps. As a strong support to our
conjecture, we prove modular invariance of (generalized) characters of the relevant (logarithmic) modules. Further evidence is provided by calculations in Zhu's algebra for $m=2$. We also present a rigorous proof of the fact that the full automorphism group of $\triplet$ is $PSL(2,\mathbb{C})$.

\end{abstract}







\section{Introduction}
Some of the most important examples of rational vertex algebras come from lattice vertex algebras (e.g. Moonshine Module).
More recently, many interesting examples of $W$-algebras have been constructed either as subalgebras of lattice vertex algebras
or as their cosets and  orbifolds.  For example, rational vertex algebras in the conjectural  $c=1$ classification  should all arise in this
way  (cf. \cite{dn} \cite{dj}). When going beyond unitary theories (easily achieved by deforming the
canonical quadratic Virasoro vector) the structure of lattice vertex algebra changes dramatically when viewed as a Virasoro algebra module,
and many "symmetries" (i.e. automorphisms) are broken.  Still, for certain values of central charge such lattice vertex algebras can have large interesting subalgebras. Triplet vertex algebra $\triplet$ is one important example; it is constructed as
the kernel of the short screening  of the rank one lattice vertex algebra with central charge $1-\frac{6(p-1)^2}{p}$. It is no longer rational, but as shown in \cite{am1}, it is $C_2$-cofinite.
We expect that $C_2$-property persists in the higher rank as well \cite{beijing} and for other vertex algebras of triplet-type \cite{am3}.


At the end of \cite{am1}, motivated by important works \cite{FGST1} \cite{FGST2} \cite{flohr}, it was pointed out that $\mathcal {W}(p)$ admits a hidden action of the Lie algebra $\frak{sl}_2(\mathbb{C})$ via derivations. Although this claim has not been proven, it is often assumed to be true due to considerations of quantum groups (cf \cite{FGST2}). Here we first present a proof of this claim for completeness \footnote{Of course, there is no natural action of ${sl}_2$ on all of $V_L$, for $p \geq 2$.} (see Theorem \ref{psl2} and Appendix). Our proof is based on the results from \cite{am1} and on an explicit construction of an automorphism $\Psi$ of order two. Because our Lie algebra acts via derivations,  its generators can exponentiate to $PSL(2,\mathbb{C})$ as a group of automorphisms of $\triplet$. Then, as we already mentioned in \cite{am1}, to each finite subgroup $\Gamma$ of $PSL(2,\mathbb{C})$, we can now associate  the orbifold subalgebra $\mathcal {W}(p)^{\Gamma}$.  The main purpose of this series of papers is to study the $C_{2}$-cofiniteness of $\mathcal {W}(p)^{\Gamma}$ and their representations. The finite subgroups of $PSL(2,\mathbb{C})$ are well known to follow the ADE classification (see \cite{do}). Up to conjugation, these are: cyclic groups, dihedral groups, and three exceptional finite subgroups (tetrahedral, octahedral and icosahedral groups).
Motivated by the the rational $c=1$ case \cite{dg},\cite{dj} and some general conjectures,  we expect that

\begin{conjecture} For each $\Gamma$ in the ADE classification, $\triplet^\Gamma$ is $C_2$-cofinite.
\end{conjecture}

This paper deals primarily with the case when $\Gamma=A_m$, that is, the automorphism group is a finite cyclic group of order $m$.
This case is easier to handle due to the fact that the oribifold algebra $\triplet^\Gamma$ can be defined as a subalgebra
of the orbifold subalgebra  $V_L^{\Gamma}$, so no $\frak{sl}_2(\mathbb{C})$ considerations are needed.
After we recall several structural results concerning the triplet vertex algebra and its orbifold $\triplet^{A_m}$ we prove

\begin{theorem}\label{de}
For $\Gamma$ of type $A_{m}$, the corresponding invariant subalgebra $\mathcal {W}(p)^{\Gamma}$ is $C_{2}$-cofinite.
\end{theorem}

After this we move on to study structure of irreducible $\am$-modules. Irreducible modules come in three families: $\Lambda$, $\Pi$ and $R$-series,
coming from the $\triplet$-modules of type $\Lambda$ and $\Pi$ and twisted $V_L$-modules, respectively. We prove the following result
\begin{theorem} \label{irrep}
Three series of irreducible modules contribute with $2pm^2$ irreducible $\am$-modules. Conjecturally, these are all irreducible modules, up to isomorphism.
\end{theorem}

Next, we consider the $m=2$ case via Zhu's algebra $A(\am)$
We give further evidence for the above conjecture based on the properties of $A(\triplet^{A_2})$.

Finally, in the last section, we compute graded characters of irreducible $\am$-modules and contemplate their modular closure.
As in \cite{am1}, we show that the space of characters of the proposed modules close under $SL(2,\mathbb{Z})$, but only after we include
certain generalized characters coming from logarithmic modules. Our result in this direction is

\begin{theorem}
The vector space generated by the characters and generalized characters coming from these three
series is $pm^2+2p-1$-dimensional. Moreover, compared to $\triplet$, no new generalized characters occur in $\am$.
\end{theorem}
This modular invariance result is a strong evidence that our classification of irreps is complete.




We finish this introduction with a few words about a possible Kazhdan-Lusztig dual quantum group of $\triplet^{\Gamma}$.
It is known that the abelian category $\overline{U}_q(sl_2)$-Mod is equivalent to
the category $\triplet$-Mod \cite{nt}. Therefore it seems natural to have an embedding
$$ \overline{U}_q(sl_2) \hookrightarrow \overline{U}_q(sl_2)^{\Gamma}, $$
where  $\overline{U}_q(sl_2)^{\Gamma}$ is yet-to-be defined finite-dimensional quantum group,
such that $\triplet^\Gamma$-Mod and $\overline{U}_q(sl_2)^{\Gamma}$-Mod are equivalent (as abelian categories).
What is puzzling to us at this stage, is the meaning of the ADE classification on the quantum group side.

In our forthcoming paper \cite{alm2} we  shall study the representation theory of the   vertex algebra $\triplet ^{D_m}$. We also plan to study orbifolds of the vertex (super)algebras introduced  in \cite{AdM-striplet} and \cite{AdM-IMRN}.

\smallskip

{\bf Acknowledgements:} We thank Ole Warnaar for useful discussion and comments regarding the constant term identities
studied in the paper.

\section{The triplet vertex algebra $\triplet$ as a $\frak{sl}_2(\mathbb{C})$--module}

We begin with some preliminaries about the triplet vertex algebra $\mathcal {W}(p)$; more details can be find in \cite{am1}.
Fix an integer $p\geq2$. Let $L=\mathbb{Z}\alpha$ be a rank one lattice with a bilinear form $\langle\cdot,\cdot\rangle$ given by
$$\langle\alpha,\alpha\rangle=2p,$$
Let $$V_{L}=\mathcal{U}(\widehat{\mathfrak{h}}_{<0})\otimes\mathbb{C}[{L}]$$ be the corresponding lattice vertex operator algebra \cite{ll},
 where $\mathfrak{\widehat{h}}$ is the affinization of $\mathfrak{h}=\mathbb{C}\alpha$ induced by the bilinear form, and $\mathbb{C}[{L}]$ is the group algebra of $L$.
In $V_{L}$ we choose the following conformal vector:
$$\omega=\frac{\alpha(-1)^{2}}{4p}\textbf{1}+\frac{p-1}{2p}\alpha(-2)\textbf{1}.$$
Thus $V_{L}$ is a Virasoro algebra module of central charge
$$c_{p,1}=1-\frac{6(p-1)^{2}}{p}.$$
We recall the following two results \cite{a,am,am1}:
\begin{lemma}\label{zz}
let $Q=e^{\alpha}_{0}:V_{L}\rightarrow V_{L}$ be the screening operator and let $n\in\mathbb{Z}_{\geq0}$,
then $Q^{i}e^{-n\alpha}\neq0$ if and only if $i\leq2n$.
\end{lemma}

Set $u_{n}^{i}=Q^{i}e^{-n\alpha}$, $i,n\in\mathbb{Z}_{\geq0}$, $i\leq2n$.
\begin{lemma}\label{1h}
Each $u_{n}^{i}$ is a singular vector of $V_{L}$ (as a Virasoro algebra module). The submodule generated by these $u_{n}^{i}$ is isomorphic to
$$\overline{V}_L \cong\bigoplus_{n=0} ^{\infty}(2n+1)L(c_{p,1}, n^{2}p+np-n).$$ Moreover $\overline{V_{L}}$ is a subalgebra of $V_{L}$.
\end{lemma}
Set the triplet vertex algebra $$\mathcal {W}(p)=\overline{V_{L}}.$$

Let $V$ be a vertex operator algebra and let
$C_{2}(V)$ be the linear span of elements of type $a_{-2}b$, $a, b\in V$. Then $V/C_{2}(V)$ has a commutative algebra structure with the multiplication
$$\overline{a}\cdot\overline{b}=\overline{a_{-1}b}$$
where $-$ is the projection from $V$ to $V/C_{2}(V)$. $V$ is said to be $C_{2}$-cofinite if $V/C_{2}(V)$ is finite-dimensional.

It was proved in \cite{am1} that $\mathcal {W}(p)$ is $C_{2}$-cofinite but irrational.
Using the $C_{2}$-cofiniteness of $\mathcal {W}(p)$ and a result of \cite{miy},
the first two authors  proved the existence of logarithmic $\mathcal {W}(p)$-modules in \cite{am1,am2}.

In the end of \cite{am1}, it was announced  out that $\mathcal {W}(p)$ admits a hidden action of $\mathfrak{sl}_{2}$ as follows:
 Set $e=Q$, $h=\frac{\alpha(0)}{p}$. Let $f\in End_{Vir}(\mathcal {W}(p))$ be the unique operator defined by
 $$fe^{-n\alpha}=0, \ \ fQ^{i}e^{-n\alpha}=-i(i-2n-1)Q^{i-1}e^{-n\alpha}, \ \ 1\leq i\leq2n.$$
One checks that $[h,e]=2e$, $[h,f]=-2f$, $[e,f]=h$.
Thus we have $$\mathcal {W}(p)=\bigoplus_{n=0} ^{\infty}W_{2n+1}\otimes L(c_{p}, n^{2}p+np-n)$$
where $W_{2n+1}$ is a $(2n+1)$-dimensional irreducible $\mathfrak{sl}_{2}$-module.
We extend the action of $e,f,h$ be letting them commute with the Virasoro generators.

It is clear that $e$ and $h$ are derivation on the vertex operator algebra $\triplet$. We shall prove in Appendix that there exists an automorphism $ \Psi \in \mbox{Aut} (\triplet) $  of order two which satisfies
$$ \Psi ( Q ^i e ^{-n \alpha} ) = \frac{(-1) ^i  i !}{ (2n -i) !} Q ^{2n-i} e ^{-n\alpha}. $$
Then the operator $f$ can be represented as
$$ f = - \Psi^{-1} Q \Psi. $$
Therefore $f$ is also a derivation on $\triplet$.
 Thus the Lie algebra $\mathfrak{sl}_2(\mathbb{C})$ acts on $\triplet$ by derivations; i.e.
\be \label{der}
x(a_n b)=(xa)_n b + a_n (xb)
\ee
for all $n \in \mathbb{Z}$, and $a,b \in \triplet$ and $x \in \frak{sl}_2$.
This property is known to hold for $p=2$ by using explicit expression for $f$.  The integration of the action of $\mathfrak{sl}_{2}$
gives an action of $PSL(2,\mathbb{C})$ on the vertex operator algebra $\mathcal {W}(p)$ as an automorphism group \cite{am1} (see \cite{dg}
for the $c=1$ case).

\begin{theorem} \label{psl2}
The group $PSL(2,\Bbb C)$ acts on $\triplet$ as an automorphism group. Moreover,
$$ \mbox{Aut} (\triplet) \cong PSL(2, \Bbb C). $$
\end{theorem}
\begin{proof}
By the above arguments, we know that $PSL(2,\Bbb C)$ is a subgroup of the full automorphism  group $\mbox{Aut} (\triplet)$.

In order to prove the converse, it suffices to construct homomorphism $\mbox{Aut}(\triplet) \rightarrow PSL(2, \Bbb C)$.
Define $$e ^1 = \frac{U^+}{2 i}, \ e ^2 = \frac{U^-}{2 }, \  e^3 = H,$$
where $U^{\pm} = 2 F \pm E. $

Let $\Phi$ be an automorphism of $\triplet$. It defines an matrix $[\Phi]  = ( a_{i,j} )$ such that
$$ \Phi (e ^i) = a_{i,1} e ^1 + a_{i,2} e ^2 + a_{i,3} e ^3.$$
One can easy see that $[\Phi] \in SO(3, \Bbb C)$. This gives a group homomorphism:
$$ G: \mbox{Aut} (\triplet)\rightarrow PSL(2, \Bbb C)\cong SO(3, \Bbb C). $$
\end{proof}


Up to conjugation, there are five classes of finite subgroups of $PSL(2,\mathbb{C})$ (cf.\cite{a1,kl}).
As there is a natural projection $PSL(2,\mathbb{C})\rightarrow SL(2,\mathbb{C})$, each subgroup $\Gamma$ of $PSL(2,\mathbb{C})$ can be lifted to the double covering group $\overline{\Gamma}$ in $SL(2,\mathbb{C})$. In the following table we describe a set of generators for each finite subgroup of $PSL(2,\mathbb{C})$ in terms of elements of $SL(2,\mathbb{C})$.
{\small
\begin{table}[htp]
\begin{tabular}{|c|c|c|c|c|}
\hline
name & $\Gamma$ & order & $\overline{\Gamma}$ & generators  \\
 \hline
cyclic group & $A_{m}$ & $m$& $\overline{A}_{m}$ & $\begin{pmatrix} e^{\frac{\pi i}{m}} & 0 \\ 0 & e^{\frac{-\pi i}{m}} \end{pmatrix} $ \\
\hline
dihedral group & $D_{m}$ & $2m$& $\overline{D}_{m}$ &
$\begin{pmatrix} e^{\frac{\pi i}{m}} & 0 \\ 0 & e^{\frac{-\pi i}{m}} \end{pmatrix} $,
$\begin{pmatrix} 0 & 1 \\ -1 & 0 \end{pmatrix} $   \\
\hline
tetrahedral group  & $T$ & 12& $\overline{T}$ &
$\begin{pmatrix} e^{\frac{\pi i}{2}} & 0 \\ 0 & e^{\frac{-\pi i}{2}} \end{pmatrix} $,
$\frac{1}{2}\begin{pmatrix} 1+i & -1-i \\ -1+i & -1+i \end{pmatrix} $  \\
\hline
octahedron group  & $O$ & $24$& $\overline{O}$ &
$\begin{pmatrix} e^{\frac{\pi i}{4}} & 0 \\ 0 & e^{\frac{-\pi i}{4}} \end{pmatrix} $,
$\frac{1}{2}\begin{pmatrix} 1+i & -1-i \\ -1+i & -1+i \end{pmatrix}$  \\
\hline
icosahedron group  & $I$ & $60$& $\overline{I}$ &
$\begin{pmatrix} e^{\frac{\pi i}{5}} & 0 \\ 0 & e^{\frac{-\pi i}{5}} \end{pmatrix} $,
$\frac{1}{\sqrt3}\begin{pmatrix} -\epsilon+\epsilon^{4}, & \epsilon^{2}-\epsilon^{3} \\ \epsilon^{2}-\epsilon^{3}, & \epsilon-\epsilon^{4} \end{pmatrix} $  \\
\hline
\end{tabular}
\\
\caption{Finite subgroups of $PSL(2,\mathbb{C})$, ($\epsilon=e^{\frac{\pi i}{5}}$)}
\end{table}
}

Let $V=\mathbb{C}^{2}$ be the standard $\mathfrak{sl}_{2}$-module. Then the $2n$-th symmetric power of $V$ is isomorphic to $W_{2n+1}$ by identifying
$x^{i}y^{2n-i}$ with $(2n-i)!Q^{i}e^{-n\alpha}$. Under this identification, for each $\phi\in PSL(2,\mathbb{C})$,
\begin{equation}\label{eqs}
 \phi (Q^{i}e^{-n\alpha})=\frac{1}{(2n-i)!}(ax+by)^{i}(cx+dy)^{2n-i}.
\end{equation}
where $\bigl(\begin{smallmatrix} a & b \\ c & d \end{smallmatrix} \bigr)$ is the lifting of $\phi$ in $SL(2,\mathbb{C})$.


\section{The algebra $\am$}

Consider the automorphism $\sigma=exp( \pi i \frac{\alpha(0)}{p m})$ acting on $V_L$, with $\langle \alpha,\alpha \rangle=2p$.
Denote by $V_L^\sigma$ the $\sigma$-fixed subalgebra of $V_L$. It is clear that $$V^\sigma_L =  \bigoplus_{n \in \mathbb{Z}}   M(1) \otimes e^{n m \alpha}.$$
Now we have
\be \label{wam}
\am = \triplet \cap V^\sigma_L.
\ee


It is relatively easy to see that the invariant subalgebra $\mathcal {W}(p)^{A_{m}}$, as a Virasoro algebra module,
is generated by $$u_{n}^{i}=Q^{i}e^{-n\alpha},\  i,n\in\mathbb{Z}_{\geq0},\  i\leq2n,\ m| (n-i).$$
Set $$F ^{(m)}=e^{-m\alpha},$$ $$E^{(m)}=Q^{2m}e^{-m\alpha},$$ $$H=Qe^{-\alpha}.$$
As in the proof of Proposition 1.3 in \cite{am1}, we quickly obtain the following result.
\begin{theorem}\label{1a}
The vertex operator algebra $\mathcal {W}(p)^{A_{m}}$ is strongly generated by $E^{(m)}, F^{(m)}, H$ and $\omega$.
\end{theorem}

For a vertex algebra $V$, let $\mathcal{P}(V) = V / C_2 (V)$ denotes the corresponding Zhu's Poisson $C_2$--algebra.
Let $M(1)=\widehat{\mathfrak{h}}_{<0}\subset V_{L}$ be the vertex operator algebra associated to $\mathfrak{\widehat{h}}$.
 Let $\overline{M(1)}=\mathcal {W}(p)\cap M(1)$. We need the following result
\begin{lemma}\label{3a}
$\mathcal{P}(M(1)))$ is the polynomial algebra generated by $\beta=\overline{\alpha(-1)\textbf{1}}$.
Let $C_{p}=\frac{(4p)^{2p-1}}{(2p-1)!^{2}}$. Then $$\mathcal{P} (\overline{M(1)})\cong\mathbb{C}[x,y]/(y^{2}-C_{p}x^{2p-1})$$
with the isomorphism sending $\overline{H}$ (resp. $\omega$) to $y$ (resp. $x$).
Moreover, the inclusion $\overline{M(1)} \hookrightarrow M(1)$ induces an injective homomorphism
$$\varphi:\mathcal{P} (\overline{M(1)}) \rightarrow \mathcal{P} (M(1))$$
such that $\varphi(\overline{\omega})=\frac{1}{4p}\beta^{2}$ and $\varphi(\overline{H})= \frac{1}{(2p-1)!} \beta ^{2p-1}$.
\end{lemma}
\begin{proof}
The structure of $\mathcal{P} (\overline{M(1)})$ can be easily deduced from the structure of Zhu's algebra $A(\overline{M(1)})$ (cf. \cite{a}). The other statements are trivial.
\end{proof}

 \begin{lemma}\label{4a}
$\overline{F ^{(m)} _{-n}E ^{(m)}}\neq 0$ in $\mathcal{P}(\overline{M(1)})$ for some $n>2$.
\end{lemma}
\begin{proof}
Let $\Delta_{s}(x_{1},\cdots, x_{s})=\prod_{i<j}(x_{i}-x_{j})$ be the Vandermonde determinant. For any $h\in\widehat{h}$ let
$$E^{\pm}(h, x_{1},\cdots, x_{s})=exp\Biggl(\sum_{n\in\pm\mathbb{Z}_{+}}\frac{h(n)}{n}(x_{1}^{-n}+\cdots+x_{s}^{-n})\Biggr).$$
In the following we abbreviate $E^{\pm}(h, x_{1},\cdots, x_{s})$ to $E^{\pm}(h, s)$, $x_{1}\cdots x_{2m}$ to $\otimes x$, and $x_{1}+\cdots+x_{2m}$ to $\oplus x$.
Direct calculations show that
{\setlength{\arraycolsep}{0pt}
\begin{eqnarray*}
&&Y(F ^{(m)} ,x_{0})E ^{(m)}  \\
&=&Res_{x_{1}}\cdots Res_{x_{2m}}(\otimes x)^{-2mp}\Delta_{2m}^{2p}Y(e^{-m\alpha},x_{0})E^{-}(-\alpha,2m)e^{m\alpha}   \\
&=&Res_{x_{1}}\cdots Res_{x_{2m}}x_{0}^{-2m^{2}p}(\otimes x)^{-2mp}\Delta_{2m}^{2p}E^{-}(m\alpha, x_{0})E^{+}(m\alpha, x_{0})E^{-}(-\alpha, 2m)\textbf{1}\\
&=&Res_{x_{1}}\cdots Res_{x_{2m}}x_{0}^{-2m^{2}p}(\otimes x)^{-2mp}\Delta_{2m}^{2p}(\prod_{i=1}^{2m}(1-\frac{x_{i}}{x_{0}})^{-2mp})E^{-}(m\alpha, x_{0})E^{-}(-\alpha, 2m)\textbf{1}.
\end{eqnarray*}
}
Assume the lemma is false. Then under the projection $$M(1)\rightarrow M(1)/C_{2}(M(1))=\mathbb{C}[\beta],$$
{\setlength{\arraycolsep}{0pt}
\begin{eqnarray*}
&&\overline{Y(F ^{(m)},x_{0})E ^{(m)}}   \\
&=&Res_{x_{1}}\cdots Res_{x_{2m}}x_{0}^{-2m^{2}p}(\otimes x)^{-2mp}\Delta_{2m}^{2p}(\prod(1-\frac{x_{i}}{x_{0}})^{-2mp})e^{\beta(\oplus x-mx_{0})}   \\
&=&\beta^{2m(mp+p-1)}f(\beta x_{0})
\end{eqnarray*}
}
for some Laurent polynomial $f$.

Set $\beta=1$ in the above formula we get
{\setlength{\arraycolsep}{0pt}
\setcounter{equation}{1}
\begin{eqnarray*}
&&e^{-mx_{0}}Res_{x_{1}}\cdots Res_{x_{2m}}x_{0}^{-2m^{2}p}(\otimes x)^{-2mp}\Delta_{2m}^{2p}(\prod(1-\frac{x_{i}}{x_{0}})^{-2mp})e^{(\oplus x)} \nonumber \\
&=&f(x_{0}).
\end{eqnarray*}
}
It is easy to see that $$g(x_{0})=Res_{x_{1}}\cdots Res_{x_{2m}}(\otimes x)^{-2mp}\Delta_{2m}^{2p}(\prod(1-\frac{x_{i}}{x_{0}})^{-2mp})e^{(\oplus x)} $$
is also a Laurent polynomial. Then the above formula forces
$g(x_{0})=0$, otherwise $$e^{-mx_{0}}x_{0}^{-2m^{2}p}g(x_{0})$$ will not be a Laurent polynomial.
But the Morris constant
term identity \cite{mo} implies that the coefficient of $x_{0}^{-p+1}$ in $g(x_{0})$ is

{\setlength{\arraycolsep}{0pt}
\begin{eqnarray*}
&&Res_{x_{1}}\cdots Res_{x_{2m}}(x_{1}\cdots x_{2m})^{-2mp}\Delta_{2m}^{2p}(\prod(1-x_{i})^{-2mp})  \\
&=&(-1)^{mp}\prod_{i=0}^{2m-1}{(-2m+i)p \choose p-1}\frac{(p-1)!((i+1)p)!}{(p-1+ip)!p!} \neq0.
\end{eqnarray*}
}
This contradicts our assumption and completes the proof.
\end{proof}

\begin{remark}The   proof of Lemma \ref{4a} implies  that  $$E^{(m)}_{n}F^{(m)}=F^{(m)}_{n}E^{(m)}=\delta_{n,2m(mp+p-1)-1}C_{m,p}\textbf{1}, $$
 where $n\geq2m(mp+p-1)-1$ and $$C_{m,p}=(-1)^{mp}\prod_{i=0}^{2m-1}{(-2m+i)p \choose p-1}\frac{((i+1)p)!}{(p-1+ip)!p}.$$
 This formula will be useful in our further analysis of subalgebras of the triplet vertex algebra $\mathcal {W}(p)$.
 \end{remark}

\begin{theorem}\label{2a}
 $\mathcal {W}(p)^{A_{m}}$ is $C_{2}$-cofinite.
\end{theorem}

\begin{proof}
As in the proof of Theorem 2.1 in \cite{am1}, we deduce from Theorem \ref{1a} that the commutative algebra $\mathcal {W}(p)^{A_{m}}/C_{2}(\mathcal {W}(p)^{A_{m}})$
is generated by $\overline{E ^{(m)} }, \overline{F ^{(m)} }, \overline{H} $ and $ \overline{\omega}$. Moreover $$\overline{E ^{(m)} }^{2}=\overline{F ^{(m)} }^{2}=0$$
 and $$\overline{H}^{2}=C_{p}\overline{\omega}^{2p-1}.$$
According to Lemma \ref{4a}, there exists some integer $n$, satisfying $n\geq 2$ and $\overline{F ^{(m)} _{-n}E ^{(m)}}\neq0$ in $M(1)/C_{2}(M(1))$.
%
This implies that
$$ \overline{F ^{(m)} _{-n}E ^{(m)}} = a \ \overline{\omega} ^k \overline{H} \quad ( k=m ^2 p + mp -m -2 + n/2, \ a \in {\Bbb C}, a \ne 0)$$
if $n$ is even,
or
$$ \overline{F ^{(m)} _{-n}E ^{(m)}} = a \ \overline{\omega} ^k   \quad (k = m  ^2 p + m p - m + (n-1)/2, \ \ a \in {\Bbb C}, a \ne 0)$$
if $n$ is odd.

 On the other hand, we have $\overline{F ^{(m)} _{-n}E ^{(m) } }=0$ in $\mathcal {W}(p)^{A_{m}}/C_{2}(\mathcal {W}(p)^{A_{m}})$ for $n\geq2$. This implies that
 $ \overline{\omega}, \overline{H}$ are also nilpotent in $\mathcal{P} (\am)$.

Therefore, $\mathcal {W}(p)^{A_{m}}/C_{2}(\mathcal {W}(p)^{A_{m}})$ is finite-dimensional.
\end{proof}

 \begin{remark}
In our forthcoming paper \cite{alm2}, we shall prove $C_2$-cofiniteness of the orbifold vertex algebra $\triplet ^{D_m}$ and classify irreducible modules of its subalgebra $\overline{M(1)} ^+$.
 \end{remark}

\section{Towards irreducible $\am$-modules}

\label{konstrukcija-modula}

There are several clues leading to classification of irreducible modules for orbifold vertex algebras.
Folklore meta-theorem in this direction says that
all $V^G$--modules should come from $V$-modules (by restrictions) and from
 $g$-twisted $V$-modules, where $g \in G$. This can be made more precise if the category of
 $V$-modules is a modular tensor category. Guided by this principle, we now investigate irreducible $\am$-modules. In this part we classify
 such irreducible modules. Conjecturally, these are all irreducible modules. Support for
 that is presented in next sections.

\subsection{Irreducible $\am$-modules: \ $\Lambda$-series}

In this part we construct $pm$ irreducible $\am$-modules starting from
$\triplet$-modules  $\Lambda(1)$,...,$\Lambda(p)$ (we use notation from \cite{am1}).
Recall that $\Lambda(i)$ has lowest conformal weight
$${\rm deg}(e^{(i-1) \alpha/2p})=\frac{(i-1)(i-1-2p+2)}{4p}=h_{i,1}  , \ \ \ i=1,...,p.$$
These modules are denoted by $\Lambda(1)$,...,$\Lambda(p)$ in \cite{am1}.
The next decomposition is well-known (cf. \cite{am1}):

$$\Lambda(i)=\bigoplus_{n=0}^\infty (2n+1) L(c_{p,1},h_{i,2n+1}),$$
where the null vectors are given by $Q^j e^{-n\alpha+(i-1)\alpha/2p}$ in the natural range.

Let now $m=2k$ (even).

Set
$$\Lambda(i)_0=\am \cdot  e^{(i-1)\alpha/2p}$$
$$\Lambda(i)^-_j=\am \cdot  e^{-j \alpha+\frac{(i-1) \alpha}{2p}}, \ \ j=1,...,k-1$$
$$\Lambda(i)^+_j=\am \cdot Q^{2j} e^{-j \alpha+\frac{(i-1) \alpha}{2p}}, \ \ j=1,...,k-1$$
$$\Lambda(i)_m=\am \cdot e^{- m\alpha/2+  \frac{(i-1) \alpha}{2p}  }$$

The next result comes immediately from the above decomposition

\begin{proposition} \label{lambda-even} Let $m=2k$. As Virasoro modules

\begin{itemize}

\item[(1)]

$$\Lambda(i)_0=\bigoplus_{n=0}^\infty (2n+1) \bigoplus_{k=0}^{m-1} L(c_{p,1},h_{i,2(nm+k)+1}),$$

\item[(2)]

$$\Lambda(i)^-_j=\Lambda(i)^+_j= \bigoplus  _{n=0}^\infty (2n+1) \bigoplus_{k=j}^{m-j-1} L(c_{p,1},h_{i,2(nm+k)+1})$$
$$+\bigoplus  _{n=0}^\infty (2n+2) \bigoplus_{k=m-j}^{m+j-1} L(c_{p,1},h_{i,2(nm+k)+1}),$$

\item[(3)]

$$\Lambda(i)_m=\bigoplus_{n=0}^\infty (2n+2) \bigoplus_{k=0}^{m-1} L(c_{p,1},h_{i,2(nm+\frac{m}{2}+k)+1}).$$

\end{itemize}

\end{proposition}

For $m=2k+1$ (odd case) the module $\Lambda(i)_m$ does not appear in the decomposition. Instead we have
$\Lambda(i)_0$ and $\Lambda(i)_j^\pm$, $j=1,...,k$ defined by the same formulas as above.

\begin{proposition} \label{lambda-odd} For $m=2k+1 \geq 1$, we have

\begin{itemize}

\item[(1)]

$$\Lambda(i)_0=\bigoplus_{n=0}^\infty (2n+1) \bigoplus_{k=0}^{m-1} L(c_{p,1},h_{i,2(nm+k)+1}),$$

\item[(2)]

$$\Lambda(i)^-_j=\Lambda(i)^+_j= \bigoplus  _{n=0}^\infty (2n+1) \bigoplus_{k=j}^{m-j-1} L(c_{p,1},h_{i,2(nm+k)+1})$$
$$+\bigoplus  _{n=0}^\infty (2n+2) \bigoplus_{k=m-j}^{m+j-1} L(c_{p,1},h_{i,2(nm+k)+1}).$$



\end{itemize}

\end{proposition}

As in \cite{am1} we infer
\begin{theorem} \label{irrL}
$\Lambda(i)^\pm_j$ are irreducible $\am$-modules. In particular, the vertex algebra $\am$ is simple.

\end{theorem}

\begin{proof} We first recall that $\Lambda(i)$ is an irreducible $\triplet$-modules. We have eigenspace decomposition
$$\Lambda(i)=\oplus_{j, \epsilon \in \pm} \Lambda(i)_j^\epsilon$$
with respect to the action of the automorphism $\sigma$. In particular we have a decomposition of $\Lambda(1)=\triplet$.
Fix $0 \neq v \in \Lambda(i)^\epsilon_j$. The we have $$\Lambda(i)={\rm Span}\{ w_n v : n \in \mathbb{Z}, \ w \in \triplet  \}.$$
But eigenvalue decomposition implies
 $$\Lambda(i)^\epsilon_j={\rm Span}\{ w_n v : n \in \mathbb{Z}, \ w \in \am \},$$
meaning that $\Lambda(i)_j^\epsilon$ are irreducible.
\end{proof}
\vskip 5mm

\subsection{Irreducible $\am$-modules:  $\Pi$-series}

Recall that for $i \in \{1,...,p \}$,
$$\Pi(i)=\triplet \cdot e^{ \frac{-p-1 + i}{2p } \alpha},$$
and
$$\Pi(i)=\bigoplus_{n=1}^\infty (2n) L(c_{p,1}, h_{i+p,2n+1}).$$

Top component of $\Pi(i)$ is spanned by $e^{ \frac{-p-1 + i}{2p } \alpha}$ and $Q e^{ \frac{-p-1 + i}{2p } \alpha}$.

Let first $m=2k$ (even).

Consider now  for $j=1,...,k$.  Let

$$\Pi(i)^-_j=\am \cdot   e^{ \frac{-p-1 + i}{2p } \alpha - (j-1) \alpha}, \ \ j=1,...,k.$$
$$\Pi(i)^+_j=\am \cdot Q^{2j-1} e^{ \frac{-p-1 + i}{2p } \alpha - (j-1) \alpha}  \ \ j=1,...,k.$$

Similarly, for $m=2k+1$ (odd case) we let
$\Pi(i)^\pm_j$ as above. In addition, we define
$$\Pi(i)_{m}= {\am} . Q^{2k+1} e^{ \frac{-p-1 + i}{2p } \alpha - k \alpha} .$$

\begin{proposition} \label{pi-even} Let $m=2k$. For $j=1,..,k$, we have decomposition of Virasoro modules:

$$\Pi(i)^-_j=\Pi(i)^+_j= \bigoplus  _{n=0}^\infty (2n+1) \bigoplus_{k=j}^{m-j} L(c_{p,1},h_{i+p,2(nm+k)+1})$$
$$+\bigoplus  _{n=0}^\infty (2n+2) \bigoplus_{k=m-j+1}^{m+j-1} L(c_{p,1},h_{i+p,2(nm+k)+1}).$$

\end{proposition}

\begin{proposition} \label{pi-odd} Let $m=2k+1$.

 For $j=1,..,k$, we have decomposition of Virasoro modules:

$$\Pi(i)^-_j=\Pi(i)^+_j= \bigoplus  _{n=0}^\infty (2n+1) \bigoplus_{k=j}^{m-j} L(c_{p,1},h_{i+p,2(nm+k)+1})$$
$$+\bigoplus  _{n=0}^\infty (2n+2) \bigoplus_{k=m-j+1}^{m+j-1} L(c_{p,1},h_{i+p,2(nm+k)+1}),$$

and
$$\Pi(i)_m=\bigoplus_{n=1}^\infty (2n) \bigoplus_{k=0}^{m-1} L(c_{p,1},h_{i,2(nm+\frac{m-1}{2}+k)+1}).$$

\end{proposition}

 As in the proof of Theorem \ref{irrL} we easily see that $\Pi(i)^\pm_j$ are all irreducible.

\subsection{Irreducible $\am$-modules: twisted series}

Recall a well-known fact.

\begin{lemma} Let $m \geq 2$. For $j=0,...,2p-1$, $i=1,...,m-1$, the space $V_{L+\frac{j-\frac{i}{m}}{2p} \alpha}$ has a $\sigma^i$-twisted $V_L$-module structure.
Moreover, $V_{L+\frac{j-\frac{i}{m}}{2p} \alpha}$  is an ordinary $V_L^{A_m}$ (and thus $\am$-module).
\end{lemma}

We immidiately get decomposition of $V_{L+\frac{j-\frac{i}{m}}{2p} \alpha}$ into
$\am$-modules: For $k=0,...,m-1$, we let
$$R(i,j,k):=\bigoplus_{s \in \mathbb{Z}} M(1) \otimes e^{\frac{j-\frac{i}{m}}{2p} \alpha+(m s+k) \alpha}$$
From the Feigin-Fuchs classification of modules, we conclude that each
summand appearing in the decomposition is irreducible as Virasoro module.

Thus we get

\begin{corollary} Each $R(i,j,k)$ is an irreducible $\am$-module. All together, twisted $V_L$-modules
yield  $2pm(m-1)$ irreducible $\am$-modules.
\end{corollary}

Next result will identify lowest weight vector in $R(i,j,k)$.

\begin{lemma} \label{parametar}  There  is a unique
  $ \ell \in \Z $ such that
$$  -(m-1) p \le \ell \le (m+1) p -1 \quad \mbox{and} \quad  e ^{ \frac{ \ell - \frac{i}{m}}{2p} \alpha} \in R(i,j,k).$$
Moreover,
 $e ^{ \frac{ \ell - \frac{i}{m}}{2p} \alpha} $ is a lowest weight vector  in $R(i,j,k)$ and
 $$ L(0) e ^{ \frac{ \ell - \frac{i}{m}}{2p} \alpha}  = h_{\ell + 1 - i/m,1} e ^{ \frac{ \ell - \frac{i}{m}}{2p} \alpha}, \
 H(0) e ^{ \frac{ \ell - \frac{i}{m}}{2p} \alpha} = { \ell - \tfrac{i}{m} \choose 2p-1 } e ^{ \frac{ \ell - \frac{i}{m}}{2p} \alpha}. $$
\end{lemma}

\subsection{Irreducible $\am$-modules: lowest weights}

Irreducible modules constructed above  are of lowest weight type with respect to $(L(0), H(0) )$.
The list of the corresponding lowest weights can be found in the following tables:

\vskip 5mm

Let $m = 2k$.
\vskip 3mm
\begin{center}
  \begin{tabular}{|c|c|c|c}
    \hline
     \mbox{module}  $M$ & \mbox{lowest weights} & $ \dim M(0) $ \\ \hline \hskip 2mm
   $\Lambda(i) _0 $ & $  ( h_{i,1}, 0)  $ & $1 $ \\ \hline
  $ \Lambda(i)_j ^ +  $  & $ ( h_{i, 2 j +1}, { - 2 j p-1 + i \choose 2p-1} ) $ & $1$ \\
   \hline
   $ \Lambda(i)_j ^ -  $  & $ ( h_{i, 2 j +1}, -{ - 2 j p-1 + i \choose 2p-1} ) $ & $1$ \\
   \hline
   $ \Lambda(i)_m  $  & $ ( h_{i, 2 k +1}, { - 2 k p-1 + i \choose 2p-1} ) $ & $2$ \\
   \hline
 $ \Pi(i)_j  ^+ $ & $ ( h_{p+i,2 j+1}, {-(2j -1) p-1 + i \choose 2 p -1} ) $ & $1$  \\ \hline
     $\Pi(i)_j  ^- $ & $ ( h_{p+i,2 j +1} , - {-(2j - 1) p-1 + i \choose 2 p -1} ) $& $1 $ \\ \hline
    $R( i, j,k)$  &  $( h_{\ell + 1 - i/m,1}, { \ell - \tfrac{i}{m} \choose 2p-1 } )$& $1$ \\
    \hline
  \end{tabular}
\end{center}

\vskip 3mm

The set of lowest conformal weights is  \bea S_m &:=&  \{ h_{i, 2j+1} \ \vert \  i = 1, \dots, p, j= 0, \dots, k \}  \cup \{ h_{p+i, 2j+1} \ \vert \  i=1, \dots, p, \ j= 1, \dots, k  \}  \nonumber \\
&& \cup \ \{  h_{\ell + 1 - i/m,1}  \ \vert \   p \le \ell \le (m+1) p -1, \ \ 1 \le i \le m-1 \} \nonumber \} \eea

\noindent Let $m = 2k+1$.

\vskip 3mm
 \begin{center}
 \begin{tabular}{|c|c|c|c}
    \hline
     \mbox{module}  $M$ & \mbox{lowest weights} & $ \dim M(0) $ \\ \hline \hskip 2mm
   $\Lambda(i) _0 $ & $  ( h_{i,1}, 0)  $ & $1 $ \\ \hline
  $ \Lambda(i)_j ^ +  $  & $ ( h_{i, 2 j +1}, { - 2 j p-1 + i \choose 2p-1} ) $ & $1$ \\
   \hline
   $ \Lambda(i)_j ^ -  $  & $ ( h_{i, 2 j +1}, -{ - 2 j p-1 + i \choose 2p-1} ) $ & $1$ \\
   \hline
   $ \Pi(i)_m  $  & $ ( h_{p+i, 2 k +3}, -{ - (2 k+1) p-1 + i \choose 2p-1} ) $ & $2$ \\
   \hline
 $ \Pi(i)_j  ^+ $ & $ ( h_{p+i,2 j+1}, {-(2j -1) p-1 + i \choose 2 p -1} ) $ & $1$  \\ \hline
     $\Pi(i)_j  ^- $ & $ ( h_{p+i,2 j+1} , - {-(2j - 1) p-1 + i \choose 2 p -1} ) $& $1 $ \\ \hline
    $R( i, j,k)$  & $ ( h_{\ell + 1 - i/m,1}, { \ell - \tfrac{i}{m} \choose 2p-1 } )$ & $1$ \\
    \hline
  \end{tabular}
  \end{center}
\vskip 3mm

The set of lowest conformal weights is  \bea S_m &:=&  \{ h_{i, 2j+1} \ \vert \  i = 1, \dots, p, j= 0, \dots, k-1 \}  \nonumber \\
&& \cup \{ h_{p+i, 2j +1} \ \vert \  i=1, \dots, p, \ j= 1, \dots, k+1  \}  \nonumber \\
&& \cup \ \{  h_{\ell + 1 - i/m,1}  \ \vert \   p \le \ell \le (m+1) p -1, \ \ 1 \le i \le m-1\} \nonumber \} \eea

Number $\ell$ is defined as in Lemma \ref{parametar}.

\begin{theorem} \label{parametri}
 The $\Lambda, \Pi$ and $R$ families contain $ 2 m ^2 p$ non-isomorphic irreducible $\am$--modules. Lowest weights of irreducible modules are
 $$ (x,y) \in {\Bbb C} ^2  \ \  \mbox{such that} \ x \in S_m, \ y ^2 = C_p P(x). $$
The set of lowest conformal weights $S_m$ has $(m^2 + 1)p$ elements.
\end{theorem}

\begin{conjecture} \label{conj-1}

The $\Lambda, \Pi$ and $R$ families of $\am$--modules provides a complete list of irreducible $\am$--modules. In particular, the vertex operator algebra has $2 m ^2 p$ irreducible modules.
\end{conjecture}

\section{Modules for $\triplet ^{A_2} $ }

\label{case-m2}

 In this section we shall give some evidence for Conjecture \ref{conj-1}. For simplicity we shall study the case $m=2$, but similar analysis can be made for general $m$.

Recall that $\triplet$ is realized as a vertex subalgebra of the lattice vertex algebra $V_{L}$, where
$$ L= \Z \alpha, \quad \langle \alpha, \alpha \rangle = 2p.$$
As a vertex algebra $\triplet  ^{{A}_2 } $ is generated by

$$ F ^{(2)} = e ^{-2 \alpha}, \ \ E ^{(2)} = Q ^4 e ^{-2 \alpha}, \ H = Q e ^{-\alpha}, \ \omega, $$
and it is realized as a vertex subalgebra of the lattice vertex algebra $V_{\Z (2 \alpha)}. $

Now we shall recall  construction $ 8 p$--irreducible modules for $\triplet ^{A_2 } $ from Section \ref{konstrukcija-modula}. First we shall start with $\triplet$--modules:
$$ \Lambda(1), \dots, \Lambda(p), \Pi(1), \dots, \Pi(p).$$

Recall that these modules can be constructed as follows:
$$ \Lambda(i) =\triplet.  e ^{ \tfrac{i-1}{2p} \alpha }, \ i=1, \dots, p; $$
$$ \Pi(i) = \triplet. e^{\tfrac{ -p-1 +i }{2p} \alpha }, i= 1, \dots, p. $$
These modules are ${\Bbb Z}_2$--graded and admits the following decomposition into $\triplet ^{{A  }_2 } $--modules:
 $$  \Lambda(i) = \Lambda(i)_0  \bigoplus \Lambda(i) _2 ,  \ \quad \Pi(i) = \Pi(i)_1 ^+ \bigoplus \Pi(i)_1 ^- ,  $$
 where
 $$ \Lambda(i)_0 = \triplet ^{{A}_2 } . e ^{\tfrac{i-1}{2p} \alpha},  \quad    \Lambda(i)_2 = \triplet ^{{A}_2 } . e ^{\tfrac{i-1}{2p} \alpha- \alpha},$$
 and
 $$ \Pi(i)_1 ^+ =  \triplet ^{{A}_2 } . e^{\tfrac{ -p-1 +i }{2p} \alpha }, \quad \Pi(i)_1 ^- =  \triplet ^{{A}_2 } . Q e^{\tfrac{ -p-1 +i }{2p} \alpha }. $$

 Next we have modules:
  $$ R(3p -j) =  \triplet ^{{A}_2 } .  e ^{\tfrac{j-1/2}{2p} \alpha}, \ \ j =-p, \dots, 3 p -1. $$
  Using parametrization from Section \ref{konstrukcija-modula} we get

  $$ R(3p -j) = \left\{ \begin{array}{cc}
                         R (1, 2p+j,1) &  \mbox{if} \ -p \le j \le -1 \\
                         R(1, j, 0) &  \mbox{if} \ 0 \le j \le 2p-1  \\
                         R (1, j-2p,1)  &  \mbox{if} \ 2p \le j \le 3p-1
                       \end{array}
                       \right.
                       $$
 It is easy to see that all the modules above are irreducible and inequivalent. Moreover, these modules are of lowest  weight type with respect to $(L(0), H(0) )$.
The list of the corresponding lowest weights can be found in the following table (compared to the previous tables here we used slightly
different $h$-parametrization):

\vskip 5mm
\begin{center}
  \begin{tabular}{|c|c|c|c}
    \hline
     \mbox{module}  $M$ & \mbox{lowest weights} & $ \dim M(0) $ \\ \hline \hskip 2mm
   $\Lambda(i) _0 $ & $  ( h_{i,1}, 0)  $ & $1 $ \\ \hline
  $ \Lambda(i) _2 $  & $ ( h_{i,3}, { - 2p-1 + i \choose 2p-1} ) $ & $2$ \\ \hline
   $ \Pi(i)_1 ^+ $ & $ ( h_{3p-i,1}, {-p-1 + i \choose 2 p -1} ) $ & $1$  \\ \hline
     $\Pi(i)_1 ^- $ & $ ( h_{3p-i,1} , - {-p-1 + i \choose 2 p -1} ) $& $1 $ \\ \hline
    $R(  j)$  & $( h_{3 p +1/2 -j,1},  { 3p -1/2 -j  \choose 2p-1} ) $ & $1$ \\
    \hline
  \end{tabular}
\end{center}

\vskip 3mm

Conjecture \ref{conj-1} says that
the set
$$\{ \Lambda(i)_0, \Lambda(i)_2, \Pi ^{\pm} (i), R(j), \quad 1 \le i \le p, \ 1 \le j \le 4p \}$$
is  a complete list of irreducible $\triplet ^{A_2 } $--modules.

\subsection{ Zhu's algebra for $\triplet ^{A_2 } $  }

\vskip 10mm

We shall present certain results and conjectures for Zhu's algebra for $\triplet ^{A_2 } $. We believe that similar results hold for general $m$. But in the case $m=2$ it is easier to verify these conjectures by using computational software (e.g. Mathematica/Maple).
We shall apply methods developed in \cite{am1}, \cite{am2}. We   omit some details which are explained in our previous papers.

 Zhu's algebra $A(\triplet ^{A_2 })$  is generated by  $$[E ^{(2)}], [F ^{(2)}], [H], [\omega]. $$

Set $H ^{(2)} = Q ^2 e ^{-2 \alpha}$ and
consider
$$E^{(2)} \circ F^{(2)} \in \overline{M(1)}.$$

 We have

 \bea
  0 &=& Q ^{4} ( F ^{(2)}  \circ F ^{(2)} ) =   E^{(2)} \circ F^{(2)} + F^{(2)} \circ E^{(2)}  \nonumber \\
 && +  4 ( Q ^3 e ^{-2 \alpha} \circ Q e ^{-2 \alpha} + Q e ^{-2 \alpha} \circ Q ^3 e ^{-2 \alpha} )  + 6 H ^{(2)} \circ H ^{(2)} \nonumber \\
 &=&  E^{(2)} \circ F^{(2)} + F^{(2)} \circ E^{(2)}  + 6 H ^{(2)} \circ H ^{(2)}  - 4  H ^{(2)} \circ H ^{(2)} \nonumber \\
 && + 4 Q (  ( Q ^2 e ^{-2 \alpha} \circ Q e ^{-2 \alpha} + Q e ^{-2 \alpha} \circ Q ^2 e ^{-2 \alpha} ) \nonumber \\
 &=& E^{(2)} \circ F^{(2)} + F^{(2)} \circ E^{(2)}  + 2 H ^{(2)} \circ H ^{(2)}   \nonumber \\
 && + 4 Q ^2 ( Q e ^{-2 \alpha} \circ Q e ^{-2 \alpha}). \nonumber
 \eea

\begin{lemma} \label{5.1}
$$Q ^2 ( Q e ^{-2 \alpha} \circ Q e ^{-2 \alpha}) \in O(\overline{M(1)} ). $$
\end{lemma}
\begin{proof}
One can easily see that
$$[Q ^2 ( Q e ^{-2 \alpha} \circ Q e ^{-2 \alpha})] = F([\omega]) * [H^{(2)}]$$
for certain polynomial $ \deg F \le 3 p -1 $. But by construction $$ Q ^2 ( Q e ^{-2 \alpha} \circ Q e ^{-2 \alpha}) \in O (\triplet ^{A_2}),$$
hence it should act trivially on the lowest components of $\triplet ^{A_2}$-modules. So we should choose $\triplet ^{A_2}$--modules on whose lowest  components $H ^{(2)} (0)$ does not act trivially. In particular we have:
$$ F( h_{i,3}) = 0,  F(h_{3p+1/2 -j,1}) =0, \quad i = 1, \dots, p, \ j = 1, \dots, 2p. $$
Therefore, we have constructed $3p$ zeros of polynomial $F$. This forces $F=0$. The proof follows.
\end{proof}
 Lemma \ref{5.1} implies:

$$ E^{(2)} \circ F^{(2)} + F^{(2)} \circ E^{(2)} = - 2 H ^{(2)} \circ H ^{(2)}  -4 Q ^2 ( Q e ^{-2 \alpha} \circ Q e ^{-2 \alpha})  \in O(\overline{M(1)}). $$
Next we notice;
$$ E ^{(2) }  \circ F ^{(2)}  - F ^{(2)} \circ E ^{(2)} \in U(Vir) Q e ^{-\alpha} + U(Vir) Q ^3 e ^{-3 \alpha} .$$

Then by using the structure of Zhu's algebra for $\overline{M(1)}$  \cite{am2},
we get
\bea \label{ev-1} &&   [E ^{(2)} \circ F ^{(2)} ] = [H] * f_1([\omega])  \in A(\overline{M(1)}) , \quad \mbox{for certain} \ \ f_1 \in {\Bbb C}[x], \eea
and $$ [ E ^{(2)} \circ F ^{(2)} ] =   0 \in A(\triplet ^{A_2}) . $$  Evaluation of relation (\ref{ev-1}) on the lowest components of $\triplet ^{A_2}$--modules yields

\bea \label{rel-1} && [H] * \prod_{ i =1} ^p ([\omega]- h_{3p - i,1}) )  ([\omega]- h_{i,3}) \prod_{j=1} ^{2 p } ([\omega]- h_{3 p +1/2 -j,1}) * s([\omega]) = 0 \eea
where $s(x)$ is a certain polynomial of degree (at most) $p-1$.

Next we notice that $\Psi  (E ^{(2)} * F ^{(2)} - F ^{(2)} * E ^{(2)} )= -  (E ^{(2)} * F ^{(2)} - F ^{(2)} * E ^{(2)})$ which implies that
$$ E ^{(2)} * F ^{(2)} - F ^{(2)} * E ^{(2)}  \in U(Vir) Q e ^{-\alpha}. $$
By using the fact that $E ^{(2)}(0)$ and $F ^{(2)} (0)$ commute on the lowest components of modules $ \Pi(i)_1 ^{\pm}$ and $R(j)$ we get
 the following formula:
\bea \label{rel-2}
&& [[E ^{(2)}], [F ^{(2)}] ] =[E^{(2)}] * [F^{(2)}]-[F^{(2)}]*[E^{(2)}] \\
&&  =  [H] *  \prod_{ i =1} ^p ([\omega]- h_{3p - i,1}) )    \prod_{j=1} ^{2 p } ([\omega]- h_{3 p +1/2 -j,1}) * r ([\omega])  \nonumber \eea
for certain polynomial $r(x)$ of degree (at most) $2 p-2$.

\begin{remark}
Since $E ^{(2)}(0)$,  $F ^{(2)} (0)$ and $H(0)$ acts non-trivially on $\triplet ^{A_2}$--modules $\Lambda(i) _2$, we get that polynomial $r(x)$ is non-trivial.
\end{remark}

Since $$ [H \circ F ^{(2)} ] =  f_2([\omega]) * [F ^{(2)} ] $$ for certain polynomial $f_2$, $\deg f_2 = p$, by evaluating this relation on the lowest components of modules $\Lambda(i) ^{-}$,  we get that
$$f_2(x) = K \ell (x), \quad \ell(x)  = \prod_{i=1} ^{p} (x - h_{i,3}). $$
Non-triviality of the constant $K$ can be obtained by the following constant term identity:

\begin{conjecture} \label{conj-const}
Let $u= e^{-\alpha} \circ Q ^3 e ^{-2\alpha} \in \overline{M(1)}$. Then $u(0)$ acts on the highest weight vector $v_{\lambda}$ of the $M(1)$--module $M(1, \lambda)$ as follows
\bea  u(0) v_{\lambda} &=&  \mbox{\rm Res}_{z,z_1,z_2,z_3}\frac{(1+z) ^{2p-1-t}  (1+z_1) ^t (1+z_2) ^t (1+z_3) ^t}{z^{2+2p} (z_1 z_2 z_3) ^{4p} } \nonumber \\
&& \cdot (1-\frac{z_1}{z}) ^{-2p} (1-\frac{z_2}{z}) ^{-2p}  (1-\frac{z_3}{z}) ^{-2p}  (z_1-z_2) ^{2p} (z_1-z_3) ^{2p} (z_2-z_3) ^{2p} v_{\lambda} \nonumber \\
&=& A_p {t+2p \choose 4p-1} {t \choose 4p-1} v_{\lambda} \qquad (A_p \ne 0), \nonumber
\eea
where, as usual, an expression $(1+x/y)^s$ is always expanded in positive  powers of $x$.

\end{conjecture}

\begin{remark} \label{rema-const}
More generally, motivated by the action of $e^{-\alpha} \circ Q ^ r e ^{-(r-1) \alpha}  $ on $\overline{M(1)}$--modules, we
consider
\be \label{strp} S(t,r,p) := \mbox{\rm Res}_{z,z_1,...,z_r}\frac{(1+z) ^{2p-1-t}  \prod_{i=1}^r (1+z_i) ^t }{z^{2+2p} (z_1 \cdots z_r) ^{2(r-1)p} }
\cdot \prod_{i=1}^r (1-\frac{z_i}{z}) ^{-2p} \Delta(z_1,...,z_r)^{2p}.
\ee
We expect there are nonzero constants  $\lambda_{p,r} $ and  $\tilde{\lambda}_{p,r}$, such that
$$S(t,r,p)=\lambda_{r,p} {t+ (r-1)p \choose 2rp-1} \prod_{i=1}^{r-2} {t+(i-1)p \choose 2ip-1}$$
$$=\tilde{\lambda}_{r,p}{t+ (r-1)p \choose 2rp-1} \prod_{i=1}^{r-2} {t+(i-1)p \choose (r-1)p-1}.$$
For example, for $r=2$, we know that $S(t,2,p)={t+p \choose 4p-1}$ and $\lambda_{2,p}= {2p \choose p}{2p-2 \choose p-1}$. This is a result proven in
\cite{am2}. For $r=3$ the statement is essentially the conjecture above.
\end{remark}

 \noindent  Because of
 $$ u = -3/2 Q^2 (H \circ F ^{(2)})-\frac{3}{2} H \circ H^{(2)},$$ and thus
 $$ [u] =-3/2 [Q^2 (H \circ F ^{(2)})] \in A(\overline{M(1)}),$$ the above conjecture implies the non-triviality of constant $K$.
Now assume that this conjecture holds. (We verified this conjecture using Mathematica up to $p \le 10$.) Then
\bea \label{rel-3} &&  \ell ([\omega]) *  [F ^{(2)}] = 0, \quad \ell (x) = \prod_{i=1} ^{p} (x - h_{i,3}). \eea
Similarly,
$$  \ell ([\omega]) *  [E ^{(2)}] = 0. $$

\begin{remark}
Assume that polynomials $r(x)$ and $s(x)$ are relatively prime. Then relations (\ref{rel-1})-(\ref{rel-3}) will imply that $[H] * h([\omega] ) = 0$ where
$$ h(x) =  \prod_{ i =1} ^p (x- h_{3p - i,1})  (x - h_{i,3}) \prod_{j=1} ^{2 p } (x- h_{3 p +1/2 -j,1}) = 0. $$
in Zhu's algebra $A(\triplet ^{{A }_2 })$. This will prove Conjecture \ref{conj-1}.

In what follows, we will see that Conjecture \ref{conj-1} holds for $p \le 5$. Similar calculations can be made by computer for small values of $p$. We checked this conjecture up to $p=10$.
\end{remark}

By using Mathematica/Maple we get a list of polynomials $s(x)$ and $r(x)$ for $p \le 5 $ (up to a scalar factor):

\vskip 5mm

{\tiny
\begin{center}
  \begin{tabular}{|c|c|c|c}
    \hline
  $ \ \  p$  & $s(x)$ & $ r(x) $ \\ \hline \hskip 2mm
 $2 $& $  17 x + 28  $ & $42 - 25 x + 10 x^2 $ \\ \hline
  $ \ \ 3 $  & $ 6006 + 937 x + 932 x^2 $ &  $60060 - 47123 x + 15897 x^2 - 2520 x^3 +  $ \\
   &  & $+ 336  x ^4 $ \\
   \hline
   $\  \ 4  $  & $  1312740 - 8809 x + 81758 x^2 + 25952 x^3 $ & $5168913750 - 4548646125 x + 1727438350 x^2 - 360026392 x^3 +
   $ \\
   &   & $+ 45686256 x^4 - 3351040 x^5 + 225280 x^6$  \\
   \hline
   $ \ \ 5 $  & $ 3480248772 - 309156003 x + 118443661 x^2     $ & $ 63145633719168 - 59216427967788 x + 24520167453753 x^2   $ \\
    $  $  & $ + 24302920 x^3 + 6427600 x^4   $ & $ -
  5855373170478 x^3 + 890653763025 x^4 - 89486430800 x^5   $ \\
    &   &  $  +
  6052956000 x^6 - 255840000 x^7 + 10400000 x^8 $ \\
    \hline
  \end{tabular}
\end{center}
 }
 \vskip 5mm
They are relatively prime.

 \vskip 5mm

\begin{conjecture} \label{conj-2}

\item[(i)] Zhu's algebra $A(\triplet ^{ A_2})$ is generated by $[\omega], [H], [E ^{(2)}], F ^{(2)}]$ which satisfy the following relations:
\bea
&&  [E ^{(2)}] ^2 = [F ^{(2)}] ^2 = 0, [H] ^2 = P([\omega]) = C_p \prod_{i=1} ^{2 p-1} ([\omega]- h_{i,1}) \nonumber \\
&& \ell ([\omega]) *  [F ^{(2)}] = \ell ([\omega] ) * [E ^{(2)}] = 0, \ \ h([\omega]) * [H] = 0. \nonumber
\eea

\item[(ii)] The center of Zhu's algebra $A(\triplet ^{ A_2})$ is isomorphic to
$$ {\Bbb C}[x ] / \langle g_{2,p}(x) \rangle $$
where  $\langle g_{2,p}(x) \rangle$ is the principal ideal generated by polynomial
$$g_{2,p}(x) =   \prod_{ i =1} ^ { 3 p - 1}  (x- h_{i, 1})   \prod_{ i =1} ^ { p}   (x- h_{i,3})  \prod_{i=1} ^{2 p } (x- h_{3 p +1/2 -j,1}). $$

\item[(iii)] Dimension of  $A(\triplet ^{ A_2})$ is  $12  p -1$.
\end{conjecture}

\begin{remark}
From this conjecture will follow that $\triplet ^{A_2}$ does not have "new" logarithmic modules, only logarithmic modules which are obtained by restriction of logarithmic modules for $\triplet$ (cf. \cite{am2}). We believe that similar statement holds for general $m$. More evidence for this statement will be presented in Section \ref{karakteri}
\end{remark}

In the simplest case  $p=2$, we get relations

\bea \label{rel-4} &&    [H]  * ([\omega]- 3) ( [\omega]- \tfrac{15}{8}) ( [\omega]-1) ( [\omega] -\tfrac{3}{8})  \nonumber \\ && ([\omega]-\tfrac{45}{32}) ([\omega]- \tfrac{21}{32}) ([\omega]- \tfrac{5}{32}) ([\omega] + \tfrac{3}{32}) (17 [\omega] + 28) = 0.   \eea

\bea \label{rel-5} &&  [[E ^{(2)}], [F ^{(2)}] ] = \nu [H]   * ( [\omega]-1) ( [\omega] -\tfrac{3}{8})  \nonumber \\  && ([\omega]-\tfrac{45}{32}) ([\omega]- \tfrac{21}{32}) ([\omega]- \tfrac{5}{32}) (42 - 25 [\omega] + 10 [\omega]^2  ) \eea

By combining (\ref{rel-3}) and (\ref{rel-5}) we get

\bea \label{rel-6} &&   [H]  * ([\omega]-3) ([\omega]- \tfrac{15}{8} ) ( ( [\omega]-1) ( [\omega] -\tfrac{3}{8})  \nonumber \\  && ([\omega]-\tfrac{45}{32}) ([\omega]- \tfrac{21}{32}) ([\omega]- \tfrac{5}{32}) ([\omega] + \tfrac{3}{32})  (42 - 25 [\omega] + 10 [\omega]^2  )  = 0. \eea

Now (\ref{rel-4}) and (\ref{rel-6}) implies that

\bea \label{rel-7} &&    [H]* ([\omega]- 3) ( [\omega]- \tfrac{15}{8}) ( [\omega]-1) ( [\omega] -\tfrac{3}{8})  \nonumber \\ && ([\omega]-\tfrac{45}{32}) ([\omega]- \tfrac{21}{32}) ([\omega]- \tfrac{5}{32}) ([\omega] + \tfrac{3}{32})  = 0.   \eea

\begin{proposition}
Conjectures  \ref{conj-1} and \ref{conj-2} holds for  $m=2$ and $p$ small ($p \le 10$).
\end{proposition}

Let us describe the structure of algebra $ {\mathcal P}( \triplet ^{A_m} ) = \triplet ^{A_m} / C_2 ( \triplet ^{A_m})$  in the case $m=p=2$. It is generated by
$ \overline{E^{(2)} }, \overline{F ^{(2)} }, \overline{H}, \overline{\omega} $. The following relations holds:

 \bea && \overline{\omega} ^ {12} = \overline{E^{(2)} }^2 =  \overline{F ^{(2)} } ^2 = 0 \nonumber \\
&& \overline{H} ^2 = \nu \overline{\omega} ^3, \ \overline{E ^{(2)}  }  \overline{ F ^{(2)} }  \in  {\mathcal P} (\overline{M(1)} ) \nonumber \\
&& \overline{\omega} ^9 \overline{H} = \overline{\omega} ^2  \overline{E^{(2)} } = \overline{\omega} ^2  \overline{F^{(2)} }= 0. \nonumber
\eea

\begin{remark}We believe that there are no further relations and therefore
 $\dim  {\mathcal P}( \mathcal{W}(2) ^{A_2} ) =25$. So  $\mathcal{W}(2) ^{A_2}$ is (most likely) new example of vertex operator algebra such that $\dim \mathcal{P} (V) > \dim A(V). $
 \end{remark}

\section{Irreducible characters}
\label{karakteri}

In this section we compute the $SL(2,\mathbb{Z})$-closure of the
character of $\am$.

Set
$$\Theta_{\lambda,k}(\tau)=\sum_{n \in \mathbb{Z}+\frac{\lambda}{2k}} q^{k n^2}$$
and
$$\partial \Theta_{\lambda,k}(\tau)=\sum_{n \in \mathbb{Z}+\frac{\lambda}{2k}} 2kn q^{k n^2}.$$

Define
$$P_{\lambda,k}(\tau):=\frac{1}{\eta(\tau)} \sum_{n \in \mathbb{Z}} (2n+1) q^{k(n+\frac{\lambda}{2k})^2}=\frac{(k-\lambda)\Theta_{\lambda,k}(\tau)+(\partial \Theta)_{\lambda,k}(\tau)}{k \eta(\tau)},$$
and
$$Q_{\lambda,k}(\tau):=\frac{1}{\eta(\tau)} \sum_{n \in \mathbb{Z}} (2n) q^{k(n+\frac{\lambda}{2k})^2}=\frac{(-\lambda)\Theta_{\lambda,k}(\tau)+(\partial \Theta)_
{\lambda,k}(\tau)}{k \eta(\tau)}.$$

Observe the relations

\bea \label{sym}
&& Q_{-\lambda,k}(\tau)=-Q_{\lambda,k}(\tau), \nonumber \\
&& Q_{2k+\lambda,k}(\tau)=Q_{\lambda,k}(\tau)-2 \frac{\Theta_{\lambda,k}}{\eta(\tau)}.
\eea

By using decomposition of $\am$ into irreducible Virasoro modules we obtain
$${\rm ch}_{\am}(\tau)$$
$$=\frac{1}{\eta(\tau)} (\sum_{n \geq 0} (2n+1) q^{p(mn+\frac{p-1}{2p})^2}-\sum_{n=1}^\infty (2n-1) q^{p(mn-(m-1)-\frac{p-1}{2p})^2})$$
$$+\frac{1}{\eta(\tau)}(\sum_{n \geq 0} (2n+1) q^{p(mn+1+\frac{p-1}{2p})^2}-\sum_{n=1}^\infty (2n-1) q^{p(mn-(m-2)-\frac{p-1}{2p})^2})$$
$$+ \cdots + \frac{1}{\eta(\tau)}(\sum_{n = 0}^\infty (2n+1) q^{p(mn+m-1+\frac{p-1}{2p})^2}-\sum_{n=1}^\infty (2n-1) q^{p(mn-\frac{p-1}{2p})^2}).$$

The first and the last term combine into (after the substitution $n \mapsto -n$ in the last term)
$$\frac{1}{\eta(\tau)} \sum_{n \in \mathbb{Z}} (2n+1) q^{pm^2(n+\frac{p-1}{2pm})^2}=
 \frac{1}{\eta(\tau)}  \sum_{n \in \mathbb{Z}} (2n+1) q^{pm^2(n+\frac{mp-m}{2pm^2})^2}=P_{mp-m,pm^2}(\tau).$$
Similarly, we combine the remaining terms and obtain


\begin{theorem} (Characters of $\Lambda(i)^\pm_j$)
For $i=1,...,p$
\be \label{lambda-0}
{\rm ch}_{\Lambda(i)_0}(\tau)=P_{m(p-i),pm^2}(\tau)+P_{m(p-i+2p),pm^2}(\tau)+\cdots+P_{m(p-i+2p(m-1)),pm^2}(\tau).
\ee
In particular,
$${\rm ch}_{\am}(\tau)=P_{m(p-1),pm^2}(\tau)+P_{m(3p-1),pm^2}(\tau)+\cdots + P_{m((2m-1)p-1),pm^2}(\tau).$$

For $j=1,...,k-1$,
\bea \label{lambda-111}
{\rm ch}_{\Lambda(i)^\pm_j}(\tau)=P_{m(p-i+2pj),pm^2}(\tau)+\cdots+P_{m(p-i+2p(m-j-1)),pm^2}(\tau) \nonumber \\
+Q_{m(p-i-2pj),pm^2}(\tau)+\cdots +Q_{m(p-i+2p(j-1)),pm^2}(\tau),
\eea
\be
{\rm ch}_{\Lambda(i)_m}(\tau)=Q_{mi-pm^2,pm^2}(\tau)+\cdots + Q_{mi+pm^2-2pm,pm^2}(\tau).
\ee

\end{theorem}

The next lemma follows easily by looking at $q$-expansion

\begin{lemma}
$$\frac{1}{m} \sum_{j=0}^{m-1} \partial \Theta_{sm+2pj,pm^2}(\tau)= \partial \Theta_{s,p}(\tau).$$
\end{lemma}

By using this lemma and formula  (\ref{sym}), we see that
formula (\ref{lambda-111}) can be rewritten as
$${\rm ch}_{\Lambda^\pm(i)_j}(\tau)=\sum_{i} \lambda_i  \frac{\Theta_{im,pm^2}(\tau)}{\eta(\tau)}+\nu \frac{\partial \Theta_{i,p}(\tau)}{\eta(\tau)}. $$
for some choice of constants $\lambda_i$ and $\nu \neq 0$.

\begin{theorem} (Characters of $\Pi(i)^\pm_j$)
For $i=1,...,p$
\be \label{lambda-00}
{\rm ch}_{\Pi(i)_m}(\tau)=Q_{(p-2m-i)m,pm^2}+\cdots + Q_{(p(m-1)-i)m,pm^2}
\ee
\bea
{\rm ch}_{\Pi(i)^\pm_j}(\tau) &=& P_{(2pj-i)m,pm^2}+\cdots+P_{(2p(m-j)-i)m,pm^2} \nonumber \\
&+&Q_{(2p-2pj-i)m,pm^2}+\cdots+Q_{(2pj-2p-i)m,pm^2}.
\eea
\end{theorem}


\begin{theorem} (Characters of $R(i,j,k)$)
We have
$${\rm ch}_{R(i,j,k)}(\tau)=\frac{1}{\eta(\tau)} \sum_{s \in \mathbb{Z}} q^{pm^2(s-\frac{m(p-1-j-2kp)+i}{2m^2p})^2},$$
which also equals
$$\frac{\Theta_{m(p-1-j-2kp)+i,m^2p}}{\eta(\tau)}.$$
\end{theorem}
Due to symmetry, we have
$${\rm ch}_{R(i,j,k)}(\tau)={\rm ch}_{R(m-i,2p-j-1,m-k)}(\tau).$$
We also have
$$\Theta_{\lambda+2m^2p,m^2p}=\Theta_{\lambda,mp^2}=\Theta_{-\lambda,mp^2}=\Theta_{2m^2p-\lambda,mp^2}.$$

Consider now the span of all twisted characters. By choosing $i,j$ and $k$ in the given range
we conclude that this space is spanned by
$$\frac{\Theta_{i,m^2p}}{\eta(\tau)},$$
where $i$ is not divisible by $m$. The total space of such characters is
$m^2 p-mp$-dimensional.


\begin{theorem} \label{modular} The modular closure of irreducible modules $\Lambda$, $\Pi$ and $R$ families (conjecturally all irreducible modules)
is $m^2p+2p-1$-dimensional
spanned by $$\frac{\Theta_{i,pm^2}(\tau)}{\eta(\tau)},$$
where $i=0,...,pm^2$,
$$ \frac{\partial \Theta_{i,p}(\tau)}{\eta(\tau)},$$
where $i=1,...,p-1$,
and
$$\frac{\tau \partial \Theta_{i,p}}{\eta(\tau)},$$
where $i=1,...,p-1$.
\end{theorem}
\begin{proof}
We have already seen that each character of $\Lambda$ and $\Pi$ type modules is expressible in terms of $\frac{\Theta_{jm, pm^2}(\tau)}{\eta(\tau)}$
and $\frac{\partial \Theta_{i,p}(\tau)}{\eta(\tau)}$, $i=1,...,p-1$. But we also know that the space of characters for the triplet vertex  algebra $\mathcal{W}(p)$ includes  $\frac{\partial\Theta_{i,p}(\tau)}{\eta(\tau)}$ and $\frac{\tau \partial\Theta_{i,p}(\tau)}{\eta(\tau)}$, so they must be included in the $SL(2,\mathbb{Z})$ closure
of $\am$. To finish the proof we only have to argue that each $$\frac{\Theta_{j,pm^2}(\tau)}{\eta(\tau)},$$
is included in the closure. Actually we will show that it lies in the span of irreducible characters. If $j$ is not divisible by $m$ this follows from consideration
of twisted modules preceding the theorem. If $j$ is divisible by $m$, then we observe that
 $$\Theta_{im,pm^2}(\tau)=\sum_{n \in \mathbb{Z}} q^{pm^2(n+\frac{im}{2pm^2})^2}=\sum_{n \in \mathbb{Z}}  q^{p(mn+\frac{i}{2p})^2}.$$
But the right hand side (when divided by $\eta(\tau)$), is a sum of characters of irreducible $\am$-modules. Finally, we use the well-known fact that the span of $\Theta_{i,pm^2}(\tau)$
is $pm^2+1$-dimensional.

\end{proof}

Based on the $m=2$ case and the known irreducible $\am$-modules we expect

\begin{conjecture}

\item[(1)] There are $2m^2p$ irreducible $\am$-modules up to isomorphism.

\item[(2)] The space of one-point functions on the torus is $m^2p+2p-1$-dimensional.


\end{conjecture}

\section{Appendix: A construction of the automorphism $\Psi$ of  $\triplet$}
\label{automorphism}

Define the following singular vectors in $\triplet$:
$$ v_{n,i} ^{\pm}= (2n-i)! Q ^i e ^{-n \alpha} \pm (-1) ^{i}! Q ^{2n-i} e^{-n \alpha} $$
where $ n \in {\N}$, $i=0, \dots, n. $

Let $\langle  v_{n,i} ^{\pm} \rangle$ denotes the irreducible Virasoro submodule generated by $v_{n,i} ^{\pm}. $
Define the following $\Z_2$--graduation on $\triplet$:
\bea
&&\triplet ^{+} = \bigoplus_{n = 0} ^{\infty} \bigoplus_{i=0} ^n  \langle v_{n,i} ^+ \rangle,  \nonumber \\
&&  \triplet ^{-} = \bigoplus_{n = 0} ^{\infty} \bigoplus_{i=0} ^n \langle v_{n,i} ^- \rangle. \nonumber
 \eea
 Clearly,
 $$ \triplet = \triplet ^+ \bigoplus \triplet ^- . $$
 Let $ \Psi \in \mbox{End} (\triplet) $ such that
 $$ \Psi \vert _ { {\triplet} ^{\pm} } = \pm \mbox{Id}. $$

Let $ H = Q e ^{-\alpha}$, $ U^- = v_{1,0} ^- = 2 F - E$, $U^+ = 2 F + E$.

Note that $\triplet$ is strongly generated by  $\omega, H, U^{\pm}$, and that $\omega, U ^+ \in \triplet ^+$ and $H, U^- \in \triplet ^-. $
\begin{proposition} $\Psi$ is an automorphism of $\triplet$:
$$ \Psi \in \mbox{Aut} (\triplet). $$
\end{proposition}
\begin{proof}
It suffices to check that the generators satisfy:
$$ H_j \triplet ^{\pm} \subset \triplet ^{\mp}, U^{+} _j \triplet ^{\pm} \subset \triplet ^{\pm}, \ U ^- _j \triplet ^{\pm} \subset \triplet^{\mp} $$
for every $j \in {\Z}. $ This will follow from Lemmas \ref{lem-1}, \ref{lem-2} and \ref{lem-3} below.
\end{proof}

\subsection{Some  lemmas}
\begin{lemma} \label{lem-1}
Let $j \in \Z$. Then
$$ H_j v_{n,i} ^{\pm} \in \langle v_{n-1,i-1} ^{\mp} \rangle + \langle v_{n,i} ^{\mp} \rangle + \langle v_{n+1,i+1} ^{\mp} \rangle. $$
\end{lemma}
\begin{proof}
First we notice that
$$ H_j Q ^i e ^{-n \alpha}  = Q ^i (H_j e ^{-n \alpha} ) - i Q ^{i-1} (E_j e ^{-n \alpha}),$$
$$ H_j Q ^{2n-i} e ^{-n \alpha}= Q ^{2n-i}  (H_j e ^{-n \alpha} ) - (2n-i) Q ^{2n -i-1} (E_j e ^{-n \alpha}).$$
Let $ \overline{u_i} \in U(Vir)$, $i=-1,0,1$ such that
$$ E_j e ^{-n \alpha} = \overline{u_{-1}} e ^{-(n-1) \alpha} + \overline{u_0} Q e ^{-n \alpha} + \overline{u_1} Q^2 e ^{-(n+1)\alpha}.$$
Let $\overline{u_i}'$, $i=0,1$ such that
$$ H_j e ^{- n \alpha} = \overline{u_0}' e ^{-n \alpha} + \overline{u_1}' Q e ^{-(n+1) \alpha}. $$
Now we shall find relations between $\overline{u_i}'$ and $\overline{u_i}$.

By applying $Q ^n$ on the previous relation we get:
\bea
 && \overline{u_0}' Q ^n e ^{-n \alpha} + \overline{u_1} ' Q ^{n+1} e ^{-(n+1) \alpha} \nonumber \\
 = &&  Q ^n (H_j e ^{-n\alpha}) = n E_j Q ^{n-1} e ^{-n \alpha} + H_j Q ^n e ^{-n \alpha} \nonumber \\
 = && n (\overline{u_{-1}} Q ^{n-1} e ^{-(n-1) \alpha} + \overline{u_0} Q ^n e ^{-n \alpha} + \overline{u_1} Q^{n+1}  e ^{-(n+1)\alpha}) +  H_j Q ^n e ^{-n \alpha}. \nonumber
\eea
Since $H_j Q ^n e ^{-n \alpha}$ does not contain component inside $\langle Q ^n e ^{-n \alpha} \rangle$  (cf. \cite{am}) we conclude

$$\overline{u_0}' Q ^n e ^{-n \alpha}  =  n \overline{u_0}  Q ^{n} e ^{-n \alpha}. $$
This proves that
$$ \overline{u_0}'  Q^i e ^{-n \alpha}  = n \overline{u_0} Q ^{i} e ^{-n \alpha} \qquad i=0, \dots, 2n. $$
Similarly,
$$ \overline{u_1}' Q ^{2n+2} e ^{-(2n +2) \alpha} = Q ^{2n+1} H_j e ^{-n \alpha} = (2n+1) E_j Q ^{2n} e ^{-n \alpha} = (2n+1) \overline{u_1} Q ^{2n+2} e ^{-(2n+2) \alpha}.$$
This implies
$$ \overline{u_1}'  Q ^i e ^{-(n+1) \alpha} = (2n+1) \overline{u_1} Q ^i e ^{- (n+1) \alpha } \qquad i=0, \dots, 2n +2.$$

We get:
 \bea  H_j v_{n,i} ^{\pm} & = &  (2n-i) ! H_j Q ^i e ^{-n \alpha} \pm (-1) ^i i! H_j Q ^{2n-i} e ^{-n\alpha} \nonumber \\
& =& \overline{u_{-1}} \left( (-i) (2n-i)!  Q ^{i-1} e ^{-(n-1) \alpha} \pm (-1) ^i i! (-(2n-i)) Q ^{2n-i-1} e ^{-(n-1) \alpha}  \right)\nonumber \\
& +& \overline{u_{0}} (n-i) \left(  (2n-i)!  Q ^{i} e ^{-n \alpha} \mp (-1) ^i i!  Q ^{2n-i} e ^{-n \alpha}  \right)\nonumber \\
& + & \overline{u_1} \left( (2n+1-i)! Q ^{i+1} e ^{-(n+1) \alpha} \pm (-1) ^i (i+1)! Q ^{2n +1-i} e ^{-(n+1) \alpha} \right) \nonumber \\
& = & - (2n-i) i \ \overline{u_{-1}} \  v_{n-1,i-1} ^{\mp} + (n-i) \ \overline{u_0} \ v_{n,i} ^{\mp} + \overline{u_1} \ v_{n+1,i+1} ^{\mp}. \nonumber
\eea
The lemma follows.

\end{proof}

\begin{lemma} \label{lem-2}
Let $j \in \Z$. Then
$$ U ^+ _j v_{n,i} ^{\pm} \in \langle v_{n-1,i-2} ^{\pm} \rangle + \langle v_{n-1, i} ^{\pm} \rangle +  \langle v_{n,i-1} ^{\pm} \rangle + \langle v_{n, i+1} ^{\pm} \rangle +  \langle v_{n+1,i} ^{\pm} \rangle + \langle v_{n+1, i+2} ^{\pm} \rangle. $$
\end{lemma}
\begin{proof}
First we notice that

\bea  e ^{-\alpha} _j  Q ^i e ^{-n \alpha} &=& Q ^i (e ^{-\alpha}_j e ^{-n \alpha} )- i Q ^{i-1} (H_j e ^{-n \alpha} ) + \frac{i (i-1) }{2} Q ^{i-2} E_j e ^{-n \alpha}, \nonumber \\
 e ^{-\alpha} _j  Q ^{2n-i}  e ^{-n \alpha} &=& Q ^{2n-i} (e ^{-\alpha}_j e ^{-n \alpha} )- (2n-i)  Q ^{2n- i-1} (H_j e ^{-n \alpha} ) + \nonumber \\ && +\frac{(2n-i) (2n-i-1) }{2} Q ^{2n -i -2} E_j e ^{-n \alpha}.\nonumber \eea
As in the previous Lemma, we define $\overline{u_i}, \overline{u_i}', \overline{u_i}''$ by
\bea
 E_j e ^{-n \alpha} &= & \overline{u_{-1}} e ^{-(n-1) \alpha} + \overline{u_0} Q e ^{-n \alpha} + \overline{u_1} Q^2 e ^{-(n+1)\alpha},\nonumber \\
 H_j e ^{- n \alpha} &=& \overline{u_0}' e ^{-n \alpha} + \overline{u_1}' Q e ^{-(n+1) \alpha}, \nonumber \\
 e ^{- \alpha} _j e ^{-n \alpha} &=& \overline{u_1}'' e ^{-(n+1) \alpha}. \nonumber
\eea
By the proof of  Lemma \ref{lem-1} we get:
\bea
&& \overline{u_0}'  Q ^i e ^{- n \alpha} = n \overline{u_0} \ Q ^i e ^{-n \alpha}, \nonumber \\
&&  \ \overline{u_1}' Q ^i e ^{-(n+1) \alpha} = (2n+1) \overline{u_1} \ Q ^i e ^{-(n+1) \alpha}, \nonumber \\
&& \ \overline{u_1}'' Q ^i e ^{-(n+1) \alpha} = (2n+1) (n+1) \overline{u_1} Q ^i e ^{-(n+1) \alpha}. \nonumber  \eea

Now we have
\bea
 && U_j ^+ v_{n,i} ^ {\pm}= \nonumber \\ && 2 (2n-i) ! e ^{-\alpha}_j Q ^i e ^{-n \alpha} \pm (-1) ^i i! E_j Q ^{2n-i} e ^{-n \alpha}   \label{rel-11} \\
 &&+ (2n-i)! E_j Q ^i e ^{-n \alpha} \pm 2 (-1) ^i i! e ^{-\alpha}_j Q ^{2n-i} e ^{-n \alpha} \label{rel-12}.
\eea
By direct calculation we see that the expression (\ref{rel-11}) is equal to
\bea
&& i (i-1) \overline{u_{-1}} \left( (2n-i)! Q ^{i-2}   e^{-(n-1) \alpha} \pm (-1) ^{i-2} (i-2)! Q ^{2n -i} e ^{-(n-1) \alpha} \right) +\nonumber \\
&& -i \overline{u_0} \left(  (2 n + 1- i )! Q ^{i-1} e ^{- n \alpha} \pm (-1) ^{i-1} (i-1) ! Q ^{2n+1-i} e ^{-n \alpha} \right) + \nonumber \\
&& \overline{u_1} \left(  (2 n+2-i)! Q ^{i} e ^{- (n+1) \alpha} \pm (-1) ^{i} i!  Q ^{2n+2-i} e ^{-n \alpha} \right) + \nonumber \\
=&& i (i-1) \overline{u_{-1}} v_{n-1,i-2} ^{\pm} - i \overline{u_0} v_{n,i-1} ^{\pm} + \overline{u_1} v_{n+1,i} ^{\pm} . \nonumber
\eea
Similarly, the expression (\ref{rel-12}) is equal to
$$  \left( (2n-i ) (2n-i-1)  \overline{u_{-1}} \ v_{n-1,i} ^{\pm} - (2n-i) \overline{u_0} \ v_{n,i+1} ^{\pm} + \overline{u_{1}} v_{n+1,i+2} ^{\pm} \right) . $$
The lemma follows.
\end{proof}

The proof of the following lemma is completely analogous to the previous case.

\begin{lemma} \label{lem-3}
Let $j \in \Z$. Then
$$ U ^- _j v_{n,i} ^{\pm} \in \langle v_{n-1,i-2} ^{\mp} \rangle + \langle v_{n-1, i} ^{\mp} \rangle +  \langle v_{n,i-1} ^{\mp} \rangle + \langle v_{n, i+1} ^{\mp} \rangle +  \langle v_{n+1,i} ^{\mp} \rangle + \langle v_{n+1, i+2} ^{\mp} \rangle. $$
\end{lemma}

\section{Appendix B}

In this section we relate our conjectural constant term identities in Conjecture \ref{conj-const} and Remark \ref{rema-const} to some known combinatorial identities in the literature \cite{k} (for an excellent review see also \cite{FW}).

Let $S(t,r,p)$ be the residue as in Remark \ref{rema-const}. We easily see
\bea
&& S(t,3,p)=\mbox{Res}_{z,z_1,z_2,z_3}\frac{(1-z) ^{2p-1-t}  (1-z_1) ^t (1-z_2) ^t (1-z_3) ^t}{z^{2+2p} (z_1 z_2 z_3) ^{4p} } \nonumber \\
&& \cdot (1-\frac{z_1}{z}) ^{-2p} (1-\frac{z_2}{z}) ^{-2p}  (1-\frac{z_3}{z}) ^{-2p}  (z_1-z_2) ^{2p} (z_1-z_3) ^{2p} (z_2-z_3) ^{2p} \nonumber \\
&& = {\rm Res}_z  \biggl\{ {\rm CT}_{z_1,z_2,z_3}  (-1)^{3p+1} \frac{(1-z)^{2p-1-t}}{z^{2+2p}} \prod_{1 \leq i < j \leq 3} (1-z_i/z_j)^p (1-z_j/z_i)^p \nonumber \\
&& \prod_{i=1}^3 (1-z_i)^{t-2p+1} \prod_{i=1}^3 (1-z_i/z)^{-2p} \prod_{i=1}^3 (1-\frac{1}{z_i})^{2p-1} \biggr\}. \nonumber
\eea

We will be using (slightly normalized) Jack polynomials $s_{\lambda}^k$ as in Kadell's paper \cite{k}. Let $k$ be a parameter
and $z=(z_1,..)$ and $y=(y_1,...)$ two sets of variables.
Then we recall the Cauchy product expansion formula for Jack polynomials (\cite{k}, \cite{FW}):

$$\prod_{i,j} \frac{1}{(1-z_i y_j)^{k}}=\sum_{\lambda} \alpha_{\lambda}^k s_\lambda^k(z) s_{\lambda}^k(y).$$
where
$$\alpha_\lambda^k=\prod_{(i,j) \in \lambda} \frac{h_{i,j}^{k,k}(\lambda)}{h_{i,j}^{k,1}(\lambda)},$$
and
$$h_{i,j}^{k,a}=( A_{i,j}(\lambda)+k L_{i,j}(\lambda)+a),$$
where $A_{i,j}(\lambda)$ and $L_{i,j}(\lambda)$ are the arm and
leg length of the partition $\lambda$ computed at the position $(i,j)$, respectively.
Now we specialize  $k=p$ (as before), $y_j=\frac{1}{z}$, with $j=1,2$. Thus we obtain a sum over partitions in at most two parts:

$$\prod_{i=1}^n \frac{1}{(1-\frac{z_i }{z})^{2p}}=\sum_{\lambda=(\lambda_1,\lambda_2)} \alpha_{\lambda}^k  s_\lambda^p(z_1,z_2,\cdots,z_n)
s_\lambda^p(\frac{1}{z},\frac{1}{z}).$$
If we specialize $n=3$, we obtain
$$\alpha^p_{\lambda_1,\lambda_2}= \frac{(p)_{\lambda_1+2p} (p)_{\lambda_2+p}(\lambda_1-\lambda_2+1)_p (\lambda_2+1)_p (\lambda_1+p+1)_p}{(\lambda_1+2p)! (\lambda_2+p)! f_3^p(\lambda)},$$
with $(x)_a=x(x+1)\cdots (x+a-1)$ the usual Pochhammer symbol and $f^p_r(\lambda)=\prod_{1 \leq i < j \leq r}(\lambda_i-\lambda_j+(i-j)p)_p$.

The next result is well-known.
 \begin{lemma}
 \[
 s^{p}_{\lambda}(z,\dots,z)=z^{|\lambda|}
 \prod_{1\leq i<j\leq n} \frac{((j-i)p+\lambda_i-\lambda_j)_p}{((j-i)p)_p}.
 \]
 Hence
 \[
 s^{p}_{\lambda}(z^{-1},z^{-1})
 =z^{-|\lambda|} \frac{(p+\lambda_1-\lambda_2)_p}{(p)_p}
 =z^{-|\lambda|} \frac{\binom{\lambda_1-\lambda_2+2p-1}{p}}{\binom{2p-1}{p}}.
 \]

 \end{lemma}

We also need the following important result, a special case of the main result of Kadell in \cite{k}:

\begin{proposition}
$${\rm CT}_{z_1,z_2,z_3}  \left\{ s^p_{\lambda_1,\lambda_2}(z_1,z_2,z_3) \prod_{i < j} (1-z_i/z_j)^p (1-z_j/z_i)^p
\prod_{i=1}^3 (1-z_i)^{t-2p+1} (1-\frac{1}{z_i})^{2p-1} \right\}$$
$$=\frac{6 f_3^p(\lambda) (-1)^{\lambda_1+\lambda_2} (t+2p)! (t+p)! t!}{(t+1+\lambda_1)!(t+1-p+\lambda_2)!(t-2p+1)!(2p-1-\lambda_1)!(3p-1-\lambda_2)!(4p-1)!}$$
$$=\frac{ 6 (2p-1)! (2p-1-\lambda_2)! f_3^p(\lambda) (-1)^{\lambda_1+\lambda_2} {t + 2p \choose  2p-1-\lambda_1}{t +p \choose 2p-1-\lambda_2}{t \choose 2p-1}}{(4p-1)! (3p-1-\lambda_2)!}$$
\end{proposition}

Putting everything together,  we obtain
$$S(t,3,p)=\sum_{\lambda_1 \geq \lambda_2 \geq 0} {\rm Res}_{z} \sum_{i=0}^\infty (-1)^i {2p-1-t \choose i} z^{i-2-2p} $$
$$ \cdot c_{\lambda_1,\lambda_2}^p (-1)^{|\lambda|} {t + 2p \choose  2p-1-\lambda_1}{t +p \choose 2p-1-\lambda_2}{t \choose 2p-1}(\frac{1}{z})^{\lambda_1+\lambda_2}$$
$$=\sum_{\lambda_1 \geq \lambda_2 \geq 0} c_{\lambda_1,\lambda_2}^p (-1)^{|\lambda|}  {t+\lambda_1+\lambda_2+1 \choose 2p+1+\lambda_1+\lambda_2} {t + 2p \choose  2p-1-\lambda_1}{t +p \choose 2p-1-\lambda_2}{t \choose 2p-1}$$

where

$$c_{\lambda_1,\lambda_2}^p=\frac{6 (2p-1)! (2p-1-\lambda_2)!  {\lambda_1-\lambda_2+2p-1 \choose p} (p)_{\lambda_1+2p} (p)_{\lambda_2+p}(\lambda_1-\lambda_2+1)_p (\lambda_2+1)_p (\lambda_1+p+1)_p}{(-1)^{p+1} (\lambda_1+2p)! (\lambda_2+p)! {2p-1 \choose p} (4p-1)! (3p-1-\lambda_2)!}.$$

\begin{remark}
By using Kadell's identity, similar identities (involving summation over partitions into two parts) can be derived for all $S(t,r,p)$.

\end{remark}

\end{document}